\documentclass[12pt]{article}
 
\usepackage[margin=1in]{geometry} 
\usepackage{amsmath,amsthm,amssymb}
\usepackage[english]{babel}
\usepackage[utf8]{inputenc}
\usepackage{tikz-cd}
\usepackage{amsmath}
\usepackage{amsmath}
\usepackage{abraces}
\usepackage{caption}
\usepackage{subcaption}
\usepackage{enumerate,latexsym}

\def\R{\mathbb{R}}

\def\Z{\mathbb{Z}}
\def\esf{\mathbb{S}}

\newcommand{\wt}{\widetilde}

\def\t{{\theta}}
\def\g{{\gamma}}
 
\newtheorem{theorem}{Theorem}

\newtheorem{proposition}[theorem]{Proposition}

\newtheorem{remark}[theorem]{Remark}

\theoremstyle{definition}
\newtheorem{definition}[theorem]{Definition}

\begin{document}
 
\begin{title}
	{Vertically invariant minimal surfaces in unimodular semidirect products}
\end{title}

\begin{author}
{David Moya \footnote{This work is supported in part by the IMAG–Maria de Maeztu grant CEX2020-001105-M / AEI /10.13039/501100011033, MICINN grant PID2020-117868GB-I00 and Junta de Andalucía grant A-FQM-139-UGR18.}}
\end{author}
\date{}

\newcommand{\Addresses}{{
  \bigskip
  \footnotesize

  \textsc{David Moya, Department of Geometry and Topology  \& IMAG, University of Granada, 18001 Granada, Spain}\par\nopagebreak
  \textit{E-mail address:} \texttt{dmoya@ugr.es}

}}

\maketitle

\begin{abstract}
A surface in a three-dimensional metric Lie group $G$ is said invariant if it is invariant with respect to a one-dimensional subgroup $\Gamma$ of the isometry group of $G$. Is this work we focus on unimodular metric Lie groups $G$ that can be written as a semidirect product of the form $\R^2\rtimes_A \R$ for certain matrix $A\in \mathcal{M}_2(\R)$ and study the minimal surfaces which are invariant under the group $\Gamma$ generated by left translations by elements in the vertical axis $\{0\}\rtimes\R$. We will call these surfaces vertically invariant. In particular, we describe new examples of minimal surfaces in $\widetilde{E}(2)$ which are vertically invariant.
\end{abstract}

\section{Introduction}	
Constant mean curvature (CMC) surfaces (including the minimal case) in homogeneous spaces is a field of interest in recent years, specially after the generalization of the holomorphic Hopf differential by Abresch and Rosenberg \cite{AbreschRosenberg} for CMC surfaces in homogeneous spaces with isometry group of dimension 4, the so-called $\mathbb{E}(\kappa,\tau)$ spaces. A non-holomorphic Hopf-type quadratic differential for CMC surfaces in three-dimensional metric Lie groups has also been developed in \cite{MeeksPerez}.

Among the family of CMC surfaces in a three-dimensional Riemannian manifold, it is natural to study and describe the surfaces that satisfy a further geometric extrinsic condition. For instance, if the ambient space $G$ is globally endowed with coordinates $(x,y,z)\in G$, we can look for surfaces given by a graph $z(x,y)=f(x)+g(y)$, where $f$, $g$ are smooth functions on some interval of $\R$ (they are called translation surfaces), as it is done in the works \cite{lop3} and \cite{Liu}; we can also impose the condition of being rotationally symmetric, as for example in \cite{Torralvo}; or, as we do in this study, we can assume that the ambient space is a metric Lie group and our surface is invariant under a one-parameter group of left translations of this metric Lie group, see \cite{loraf1},\cite{loraf2},\cite{stefano}. It is also natural to impose these geometric conditions together with constant Gaussian curvature, see \cite{torralbo2}. Imposing these conditions usually leads to specific examples of surfaces with the desired properties.

The background spaces of this work are metric Lie groups. We will call metric Lie group to a simply connected 3-dimensional Lie group endowed with a left invariant metric. On a Lie group $G$ there is a measure invariant under left translations which is called the Haar measure. When the Haar measure is also invariant under right translations, the Lie group $G$ is called unimodular. See \cite{MeeksPerez} for a complete description of unimodular and non-unimodular metric Lie groups. A particularly interesting algebraic condition on metric Lie groups is the property of being written as a semidirect product. Semidirect products comprise all the non-unimodular metric Lie groups and among the unimodular family there are only two cases that can not be written as semidirect products: the special unitary group $SU(2)$ and the universal cover $\widetilde{SL}(2,\R)$ of the special linear group.

We organise this work as follows. In Section \ref{sec2} we give some preliminaries and introduce notation involving metric Lie groups, focusing on the properties we need to develop the remaining sections. In Section \ref{sec3} we get an ODE for the horizontal curve which generates, through vertical translations, an invariant CMC surface. In Section \ref{sec4} we particularize this equation for CMC surfaces when the metric Lie group is unimodular and can be written as a semidirect product. Finally, in Sections \ref{sec5} and \ref{sec6} we describe the vertically invariant subclass of minimal surfaces in the Heisenberg group ${\rm Nil}_3$ and in the universal cover $\widetilde{E}(2)$ of the group of orientation-preserving rigid motions of the Euclidean plane. When the ambient space is $\R^3=\R^2\rtimes_A\R$ with $A=0\in \mathcal{M}_2(\R)$, it is well-known that vertical planes are the only vertically invariant minimal surfaces. When $G={\rm Sol}_3$, the vertically invariant minimal  surfaces were studied in \cite{loraf1} and \cite{loraf2}. From another point of view, the vertically invariant CMC surfaces of the Heisenberg group $\text{Nil}_3$ were classified in \cite{FMP}. In Section \ref{sec7}, we also describe vertically invariant surfaces with zero Gaussian curvature in these ambient spaces.

\section{Preliminaries on metric Lie groups}\label{sec2}
The results in this preliminary section are detailed in \cite{MeeksPerez}. For a matrix $A\in \mathcal{M}_2(\R)$ of the form 
\begin{equation*}
A=\left(\begin{array}{cc}
a & b \\
c & d
\end{array}\right),
\end{equation*}
where $a,b,c,d\in \R$, we can consider the semidirect product  $G=\R^2\rtimes_A \R$. 

Consider $(p_1,z_1),(p_2,z_2)\in \R^2\rtimes_A \R$, the operation $*$ of the semidirect product $G$ is given by:
\begin{equation}\label{6}
(p_1,z_1)*(p_2,z_2)=(p_1+e^{z_1A}p_2,z_1+z_2).    
\end{equation}

We choose global coordinates $(x,y)\in \R^2$, $z\in \R$ so that $\partial_x=\frac{\partial}{\partial x}, \partial_y=\frac{\partial}{\partial y}, \partial_z=\frac{\partial}{\partial z}$ generates the space of differentiable vector fields on $G$, $\mathfrak{X}(G)$, and denote
\begin{equation*}
e^{zA}=\left(\begin{array}{cc}
a_{11}(z) & a_{12}(z) \\
a_{21}(z) & a_{22}(z)
\end{array}\right). 
\end{equation*}
With this notation, the set $\{E_1,E_2,E_3\}$ given by:
\begin{equation}\label{3}
E_1(x,y,z)=a_{11}(z)\partial_x +a_{21}(z)\partial_y,\quad E_2(x,y,z)=a_{12}(z)\partial_x +a_{22}(z)\partial_y, \quad E_3=\partial_z
\end{equation}
is a basis of the Lie algebra $\mathfrak{g}$ of $G$. We consider the canonical left invariant metric $\langle,\rangle$ on $\R^2\rtimes_A \R$, which is the one that extends via left translations the usual inner product on $T_eG$, $e=(0,0,0)$ or equivalently for which the left invariant basis $\{E_1,E_2,E_3\}$ is orthonormal. $\langle,\rangle$ takes the following form when we consider the basis $\{\partial_x,\partial_y, \partial_z\}$:
\begin{eqnarray}
\langle ,\rangle  &=&[a_{11}(-z)^2 + a_{21}(-z)^2]\, dx^2 
+[a_{12}(-z)^2+a_{22}(-z)^2]\, dy^2 + dz^2\nonumber
\\
& & 
+[a_{11}(-z)a_{12}(-z) + a_{21}(-z)a_{22}(-z)] (dx\otimes dy 
+ dy\otimes dx).\label{3'}
\end{eqnarray}
The Levi-Civita connection $\nabla$ for the canonical metric of $G$ is determined by:
\begin{equation}\label{4}
\begin{array}{l|l|l}
\displaystyle \nabla_{E_1}E_1=a E_3  &  \displaystyle\nabla_{E_1}E_2=\frac{b+c}{2} E_3 & \displaystyle\nabla_{E_1}E_3=-aE_1 -\frac{b+c}{2} E_2\\
\displaystyle\nabla_{E_2}E_1=\frac{b+c}{2} E_3 & \displaystyle\nabla_{E_2}E_2=d E_3 & \displaystyle\nabla_{E_2}E_3=-\frac{b+c}{2} E_1-d E_2\\
\displaystyle\nabla_{E_3}E_1=\frac{c-b}{2} E_2 & \displaystyle\nabla_{E_3}E_2=\frac{b-c}{2} E_1 & \displaystyle\nabla_{E_3}E_3=0.
\end{array}  
\end{equation}
Moreover, we also need the following right invariant basis of vector fields of $G$ (in particular, these are killing vector fields):
\begin{equation}\label{5}
F_1=\partial_x,\quad F_2=\partial_y, \quad F_3(x,y,z)=(ax+by)\partial_x +(cx+dy)\partial_y+\partial_z.  
\end{equation}
In this context, we are focused on studying CMC surfaces which are invariant under certain 1-parameter subgroups of left translations. In our case we choose vertical translations. We begin by considering a surface $\Sigma$ which is generated by a curve $\gamma\subset \R^2\rtimes_A\{0\}$, $\gamma(t)=(x(t),y(t),0)$ via left translations by elements of the form $(0,0,z)$, $z\in \R$. A parametrization of $\Sigma$ can be given by $\Phi(t,s)=\phi_s(\gamma(t))$ where 
\begin{equation}\label{7}
\phi_s(x,y,z)= (0,0,s)*(x,y,z)=(e^{s A}(x,y)^T,s+z),
\end{equation}
where $(x,y)^T$ is the transpose of the row $(x,y)$. 

\begin{definition}
{\rm 
If $\g$ is a curve of the form $\gamma(t)=(x(t),y(t),0)$, we will call {\it vertically invariant surface} to the image through the parametrization $\Phi(t,s)=\phi_s(\gamma(t))$.
}
\end{definition}

\section{The mean curvature of an invariant surface through vertical translations}\label{sec3}
Let $\Sigma$ be a vertically invariant surface. We next compute the coefficients of the first and second fundamental forms of $\Sigma$, and relate them to the geometry of $\gamma$.
\begin{equation}\label{8}
\Phi_t(t,0)=\gamma'(t)=\left(\begin{array}{c}
x'(t)  \\
y'(t)  \\
0
\end{array}\right)^T=\left[\begin{array}{c}
x'(t)  \\
y'(t)  \\
0
\end{array}\right]^T
\end{equation}
where we are using parenthesis for coordinates with respect to $\{\partial_x,\partial_y,\partial_z\}$ and brackets for coordinates with respect to $\{E_1,E_2,E_3\}$. We impose that $\gamma$ is arc length parameterized (i.e. $(x')^2+(y')^2=1$, observe that $\langle,\rangle$ coincides with the usual inner product of $\R^3$ along $\R^2\rtimes_A\{0\}$), and let  $\theta=\theta(t)$ be a differentiable function such that 
\begin{equation}\label{9}
 x'(t)=\cos \theta(t), \quad y'(t)=\sin\theta(t).
\end{equation}
In order to compute $\Phi_s$, we use that $\Phi_s(t,s)=(F_3)_{\Phi(t,s)}$:
\begin{equation}\label{10}
    F_3(x,y,z)=\left(\begin{array}{c}
\delta  \\
\varepsilon\\
1
\end{array}\right)^T=\left[\begin{array}{c}
\delta a_{11}(-z)+\varepsilon a_{12}(-z)\\
\delta a_{21}(-z)+\varepsilon a_{22}(-z)\\
1
\end{array}\right]^T,
\end{equation}
(see equation 5.9 in \cite{MMP}) where $\delta(x,y)=ax+by$, $\varepsilon(x,y)=cx+dy$. Thus,
\begin{equation}\label{11}
    \Phi_s(t,0)=(F_3)_{\gamma(t)}=\left[\begin{array}{c}
 \delta(t)\\
\varepsilon(t)\\
1
    \end{array}\right]^T,
\end{equation}
where $\delta(t)=\delta(\gamma(t))$ and 
$\varepsilon(t)=\varepsilon(\gamma(t))$.
With $\Phi_t$ and $\Phi_s$ we can compute the coefficients of the first fundamental form:
\begin{equation}\label{12}
E=(x')^2+(y')^2=1,\quad
F=\delta(t)\cos\theta(t)+\varepsilon(t)\sin\theta(t),\quad 
G=1+\delta(t)^2+\varepsilon(t)^2.    
\end{equation}
Using the expressions for $\nabla_{E_i}E_j$, $i,j=1,2,3$ given in \eqref{4} we compute:

\begin{align}
    \begin{split}\label{14}
    \cdot&\nabla_{\Phi_t}\Phi_t(t,0)=\frac{D}{dt}
    \left(\cos\theta(t){E_1}_{\gamma(t)}
    +\sin\theta(t){E_2}_{\gamma(t)}\right)\\
    &=-\theta'(t)\sin\theta(t)E_1
    +\theta'(t)\cos\theta(t)E_2+\cos\theta(t)\nabla_{\gamma'(t)}E_1
    +\sin\theta(t)\nabla_{\gamma'(t)}E_2\\
    &=\left[\begin{array}{c}
    -\theta'(t)\sin\theta(t) \\
    \theta'(t)\cos\theta(t) \\
    0
    \end{array}\right]^T
    +\cos^2\theta(t)aE_3
    +\cos\theta(t)\sin\theta(t)\frac{b+c}{2}E_3\\
    &+\sin\theta(t)\cos\theta(t)\frac{b+c}{2}E_3
    +\sin^2\theta(t)d E_3\\
    &=\left[\begin{array}{c}
    -\theta'(t)\sin\theta(t) \\
    \theta'(t)\cos\theta(t) \\
    a \cos^2\theta(t) + d \sin^2\theta(t) + 
    \cos\theta(t) \sin\theta(t) (b + c)
    \end{array}\right]^T
    \end{split}\\[4ex]
    \begin{split}\label{15}
    \cdot&\nabla_{\Phi_t}\Phi_s(t,0)
    =\frac{D}{dt}(\delta(t){E_1}_{\gamma(t)}
    +\varepsilon(t){E_2}_{\gamma(t)}+{E_3}_{\gamma(t)})\\
    &=\left[\begin{array}{c}
    \delta'(t)\\
    \varepsilon'(t)\\
    0
    \end{array}\right]^T
    +\delta(t)\nabla_{\gamma'(t)}E_1
    +\varepsilon(t)\nabla_{\gamma'(t)}E_2+\nabla_{\gamma'(t)}E_3\\
    &=\left[\begin{array}{c}
    \delta'(t)\\
    \varepsilon'(t)\\
    0
    \end{array}\right]^T
    +\delta(t) \cos\theta(t)aE_3
    +\delta(t)\sin\theta(t)\frac{b+c}{2}E_3
    +\varepsilon(t)\cos\theta(t)\frac{b+c}{2}E_3\\
    &+\varepsilon(t)\sin\theta(t) dE_3+ \cos\theta(t)\left(-aE_1-\frac{b+c}{2}E_2\right)
    +\sin\theta(t)\left(-\frac{b+c}{2}E_1-dE_2\right)\\
    &=\left[\begin{array}{c}
    \sin\theta(t)\frac{b-c}{2} \vspace{0.3cm}\\
    \cos\theta(t) \frac{c-b}{2}\vspace{0.3cm}\\
    \delta(t)\left(\cos\theta(t)a+\sin\theta(t)\frac{b+c}{2}\right)
    +\varepsilon(t)\left(\cos\theta(t)\frac{b+c}{2}+\sin\theta(t)d
    \right)
    \end{array}\right]^T
    \end{split}
\end{align}
\begin{align}
    \begin{split}\label{16}
    \cdot&\nabla_{\Phi_s}\Phi_s(t,0)
    =\left(\nabla_{\delta(t)E_1+\varepsilon(t)E_2+E_3}\delta(t)E_1
    +\varepsilon(t)E_2+E_3\right)(t,0)\\
    &=\delta(t)^2aE_3
    +\delta(t)\varepsilon(t)\frac{b+c}{2}E_3
    +\delta(t)\left(-aE_1-\frac{b+c}{2}E_2\right)
    +\varepsilon(t)\delta(t)\frac{b+c}{2}E_3
    +\varepsilon(t)^2 dE_3\\
    &+\varepsilon(t)\left(-\frac{b+c}{2}E_1-dE_2\right)
    +\delta(t)\frac{c-b}{2}E_2+\varepsilon(t)\frac{b-c}{2}E_1\\
    &=\left[\begin{array}{c}
    -c\varepsilon(t)-a\delta(t)\\
    -b\delta(t) -d\varepsilon(t)\\
    \delta(t)\varepsilon(t)(b+c)+a\delta(t)^2+d\varepsilon(t)^2
    \end{array}\right]^T
    \end{split}
\end{align}
A unitary normal field to $\Sigma$ is
\begin{equation}\label{13}
N=\frac{\Phi_t\times \Phi_s}{|\Phi_t\times \Phi_s|}
=\frac{\sin\theta E_1-\cos\theta E_2
+(\varepsilon \cos\theta -\delta \sin\theta )E_3}
{\sqrt{1+(\varepsilon \cos\theta -\delta \sin\theta )^2}}.
\end{equation}

With these computations the coefficients of the second fundamental form of $\Sigma$ at $(t,0)$ become:
\begin{equation}\label{17}
e=\langle N,\nabla_{\Phi_t}\Phi_t\rangle
=\frac{(\delta \sin \theta -\varepsilon  \cos \theta) [-(a-d) \cos (2 \theta)-a-(b+c) \sin (2 \theta)-d]
-2 \theta '}
{2\sqrt{1 + (\sin\theta \delta - \cos\theta \varepsilon)^2}}
\end{equation}
\begin{equation}\label{18}
 f=\langle N,\nabla_{\Phi_t}\Phi_s\rangle
 =\frac{(\varepsilon \cos \theta -\delta  \sin \theta )
[\cos\theta (2 a \delta +(b+c) \varepsilon )+\sin \theta  ((b+c) \delta +2 d \varepsilon)]
+b-c}
{2 \sqrt{1+(\delta  \sin \theta-\varepsilon \cos \theta)^2}}
\end{equation}
\begin{equation}\label{19}
g=\frac{\cos \theta \left[a \delta ^2 \varepsilon
+\delta \left((b+c) \varepsilon^2+b\right)
+d \varepsilon \left(\varepsilon^2+1\right)\right]
-\sin \theta \left[\delta \left(a \delta^2+a
+(b+c) \delta \varepsilon+d \varepsilon ^2\right)
+c \varepsilon\right]}
{\sqrt{1+(\delta \sin \theta-\varepsilon 
\cos\theta)^2}}
\end{equation}
The mean curvature of $\Sigma$ is given by  
\begin{equation*}
    H=\frac{Eg-2Ff+eG}{2(EG-F^2)},
\end{equation*}
so imposing constant mean curvature $H\in \R$ for $\Sigma$ is equivalent to the following system of ODEs:
\begin{equation}
\left\lbrace\begin{array}{l}
x'=\cos \theta\\
y'=\sin\theta\\
\eqref{21}
\end{array},\right.
\label{sistema}
\end{equation}
where~\eqref{21} is the ODE
\begin{equation}\label{21}
\begin{split}
&8 \left[((c x+d y)\cos \theta -(a x+b y)\sin \theta )^2+1\right]^{3/2}H
=-4 \theta' \left[1+(a x+b y)^2+(c x+d y)^2\right]
\\
&-\sin \theta [(a x+b y) \left[ 3 (a+d) \left((a x+b y)^2+(c x+d y)^2\right)+5 a+3 d\right]
+(3 b-c) (c x+d y)]
\\
&+\sin (3 \theta ) [(a x+b y) \left[ (a+d) \left((a x+b y)^2-3 (c x+d y)^2\right)-a+d\right]
+(b+c) (c x+d y)]
\\
&+\cos \theta [ (c x+d y) \left[3 (a+d) (a x+b y)^2+3 a+5 d\right]
-(b-3 c) (a x+b y)+3 (a+d) (c x+d y)^3]
\\
&+\cos (3 \theta) [-(c x+d y) \left[3 (a+d) (a x+b y)^2-a+d\right]
+(b+c) (a x+b y)+(a+d) (c x+d y)^3].
\end{split}
\end{equation}

\section{Unimodular semidirec products}\label{sec4}
Our goal is to study the solutions of the system~\eqref{sistema} depending of the ambient space (or equivalently depending of the matrix $A\in \mathcal{M}_2(\R)$). We are going to focus on unimodular metric Lie groups. Using Theorem 2.15 of \cite{MeeksPerez} we know that, after scaling the metric, the matrix $A$ can be chosen as:
\[
A=0,\quad A=\left(\begin{array}{cc}
0 & \pm c \\
\frac{1}{c} & 0
\end{array}\right), \quad c\in [1,\infty), \quad \text{or} \quad A=\left(\begin{array}{cc}
0 & 1 \\
0 & 0
\end{array}\right),
\]
in the sense that a unimodular metric Lie group $G$ is isomorphic and isometric to $\R^2\rtimes_A\R$ with its canonical metric. Depending on the chosen matrix $A$, we find one of the following metric Lie groups: $\R^3$, the solvable group ${\rm Sol}_3$ of orientation-preserving rigid motions of the Lorentz-Minkowski plane, the universal cover $\widetilde{E}(2)$ of the group of orientation-preserving rigid motions of the Euclidean plane or the Heisenberg group ${\rm Nil}_3$. Next, we write the equation \eqref{21} in these ambient spaces (we do not include here $\R^3$ since the only vertically invariant surfaces in $\R^3$ are vertical planes):
\begin{itemize}
\item If we consider $\text{Sol}_3$ with the metric that makes it isometric and isomorphic to $\R^2\rtimes_{A(c)}\R$ with $A(c)=\left(\begin{array}{cc}
	0 & c \\
	1/c & 0
\end{array}\right)$ with $c\in [1,\infty)$, then \eqref{21} writes as
\begin{equation}\label{solgeneral}
H=\frac{((1+c^2)\cos(2\theta)+1-c^2)(c^2 y \cos \theta+x \sin\theta)-2\theta'(c^2+x^2+c^4 y^2)}{4c^2\left(1+\left(\frac{x \cos{\theta}}{c}-c y \sin\theta\right)^2\right)^{\frac{3}{2}}}.
\end{equation}

\item If  $G=\text{Sol}_3$ with its standard metric, then $G$ is isometric and isomorphic to $\R^2\rtimes_A\R$ 
 with $A=\left(\begin{array}{cc}
0 & 1 \\
1 & 0
\end{array}\right)$ (this is equivalent to choose $A$ as in the previous item with $c=1$). Equation \eqref{21} writes as:
\begin{equation}
H=\frac{\cos(2\theta)(x\sin\theta 	+y\cos\theta)
-(1+x^2+y^2)\theta'}
{2(1+(x\cos\theta-y\sin\theta)^2)^{\frac{3}{2}}}.
\end{equation}
We are not going to deal with vertically invariant minimal surfaces in $\text{Sol}_3=\R^2\rtimes_{A(1)}\R$ since they have already been studied in \cite{loraf1}.
    
\item If $G=\text{Nil}_3$ with its standard metric, which is isometric to $\R^2\rtimes_A\R$ with $A=\left(\begin{array}{cc}
0 & 1 \\
0 & 0
\end{array}\right)$, then equation \eqref{21} takes the form
\begin{equation}\label{nil}
H=-\frac{y\cos\theta \sin^2\theta +(1+y^2 )\theta'}
{2(1+y^2\sin^2\theta)^{\frac{3}{2}}}.
 \end{equation}
    
\item If we consider $\widetilde{E}(2)$ with the metric that makes it isometric and isomorphic to $\R^2\rtimes_{A(c)}\R$ with $A(c)=\left(\begin{array}{cc}
	0 & -c \\
	1/c & 0
\end{array}\right)$ with $c\in [1,\infty)$, then \eqref{21} writes as
\begin{equation}\label{26}
H=\frac{\left[\left(c^2-1\right) \cos (2 \theta )-c^2-1\right] 
\left(c^2 y \cos \theta -x \sin \theta \right)-2 \theta ' \left(c^2+x^2+c^4 y^2\right)}
{4\left[\left(x \cos \theta+c^2 y \sin \theta\right)^2+c^2\right] \sqrt{1+\left(\frac{x \cos \theta }{c}+c y \sin \theta \right)^2}}.
\end{equation}
\end{itemize}

\section{Vertically invariant minimal surfaces in $\text{Nil}_3$}\label{sec5}
In order to keep computations simple, we will focus on minimal surfaces, which means that we will impose $H=0$. If $\Sigma$ is minimal then system \eqref{sistema} can be written as:
\begin{equation}\label{sistema24}
\left\lbrace\begin{array}{l}
x'=\cos \theta\\
y'=\sin\theta\\
\theta'=\frac{-y\cos\theta \sin^2\theta }{1+y^2}
\end{array}.\right.
\end{equation}
The next proposition is clear by killing the previous numerator. 
\begin{proposition}\label{prop1}
The only solutions of the system \eqref{sistema24} that are obtained by choosing $\theta$ as a constant are given by 
\begin{equation}\label{sol1}
(x(t),y(t),\theta)=(t,0,0)+(x(0),y(0),0),
\end{equation}
\begin{equation}\label{sol2}
    (x(t),y(t),\theta)=(0,t,\pi/2)+(x(0),y(0),0).
\end{equation}
\end{proposition}
\begin{proposition}\label{prop2}
If $\left\lbrace x,y,\theta\right\rbrace$ is a solution of \eqref{sistema24} with a nonconstant $\theta$, then $\theta$ never attains the values $\left\lbrace \frac{k\pi}{2},\, k\in \Z\right\rbrace$.
\end{proposition}
\begin{proof}
Suppose that $\left\lbrace x,y,\theta\right\rbrace$ is a solution of \eqref{sistema24} with a nonconstant $\theta$ which satisfies that $\theta\in \left\lbrace \frac{k\pi}{2},\, k\in \Z\right\rbrace$ at $t_0$. We can consider the following solution of \eqref{sistema24}: 
\[
\left\lbrace x(t_0)+\cos\left(\frac{k\pi}{2}\right)(t-t_0), y(t_0)+\sin\left(\frac{k\pi}{2}\right)(t-t_0),\frac{k \pi}{2}\right\rbrace.
\]
Due to the uniqueness of the initial value problem associated to \eqref{sistema24}, we have $\theta(t)=\frac{k\pi}{2}$ which is a contradiction.
\end{proof}

\begin{remark}
{\rm From Proposition~\ref{prop2} we have that the generating curve $\gamma$ of a solution of \eqref{sistema24} that is not in Proposition \ref{prop1} is a graph over the axes $x$ and $y$. 
Moreover, since \eqref{sistema24} implies that the derivatives of the functions $x$, $y$ and $\theta$ are bounded, then an application of the Picard-Lindelöf Theorem gives that the maximal interval of definition of $\gamma$ is $\R$.}
\end{remark}

In order to describe the solutions of \eqref{sistema24} with a nonconstant $\theta$, we can assume that $\theta(t)\in (0,\pi/2)$ for all $t\in \R$. We give a first integral for the system \eqref{sistema24}:
\[
J(t):=(1+y(t)^2)\tan^2\theta(t),
\]
which means that $J(t)$ is constant along any solution of \eqref{sistema24}. We rewrite $J$ as
\[
J=(1+y^2)\frac{y'^2}{1-y'^2},
\]
which makes solving \eqref{sistema24} equivalent to solve 
\begin{equation}\label{sistemanilbis}
\left\lbrace\begin{array}{l}
     x'=\cos\theta \\
     y'=\sin\theta \\
     (1+y^2)\frac{y'^2}{1-y'^2}=a\in [0,\infty)
\end{array}.\right. 
\end{equation}
We now define the following diffeomorphism from $\R$ to $\R$:
\[
\begin{array}{cccl}
 f\colon & \R&\rightarrow &\R \\
& t &\mapsto & 
\frac{1}{2}(1+a)\ln(t+\sqrt{1+a+t^2})+\frac{1}{2}t\sqrt{1+a+t^2},
\end{array}
\]
where $a\geq 0$. The last equation of \eqref{sistemanilbis} can be solved using direct integration:
\begin{equation}\label{yt}
y(t)=f^{-1}(\sqrt{a}t+c_1), \quad c_1\in \R.
\end{equation}
If $a=0$, from \eqref{sistemanilbis} we deduce that the solution of this system is given by \eqref{sol2}. Otherwise we can explicitly integrate the function $x$ in \eqref{sistemanilbis}:
\begin{equation}\label{xt}
x(t)=c_2+\frac{1}{2\sqrt{a}}\left(\text{arcsinh}
(f^{-1}(\sqrt{a}t+c_1))
+f^{-1}(\sqrt{a}t+c_1)\sqrt{1+(f^{-1}(\sqrt{a}t+c_1))^2}\right),
\end{equation}
with $c_1,c_2\in \R$. These expressions describe all vertically invariant minimal surfaces in $\text{Nil}_3$. 
\begin{remark}
{\rm
\begin{enumerate}

\begin{figure}[h]
\centering
\includegraphics[scale=0.7]{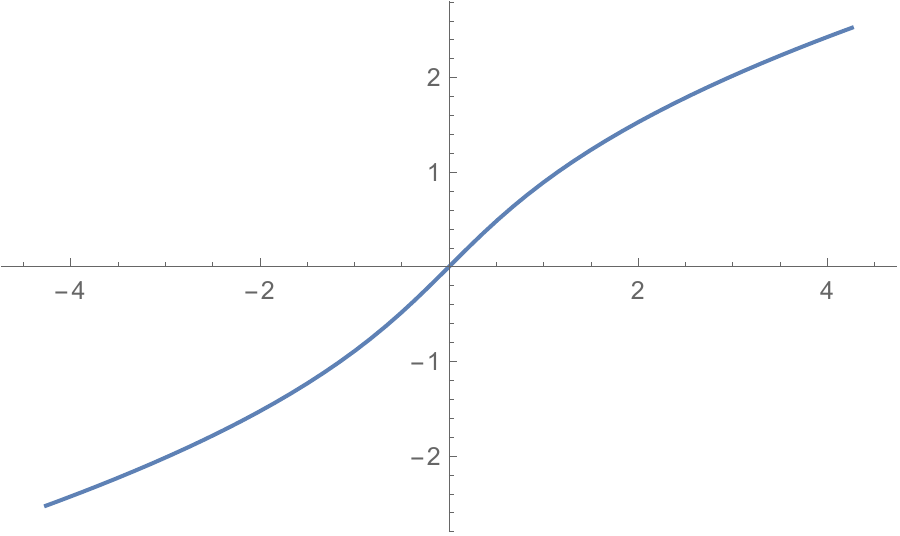}
\caption{The solution $\g$ of~\eqref{sistemanilbis} with the initial conditions $\g(0)=(0,0)$, $a=1$.}
\label{fig4}
\end{figure}

\item These surfaces are well-known; in fact CMC surfaces in ${\rm Nil}_3$ invariant under any one-dimensional subgroup of the isometry group of ${\rm Nil}_3$ were classified by Figueroa-Mercuri-Pedrosa in~\cite{FMP}. To be more precise, the surfaces given through \eqref{yt} and \eqref{xt} are described in Theorem 6 of \cite{FMP}, as we explain next. 

In \cite{FMP} Heisenberg space ${\rm Nil}_3$ is given as $\R^3$ with the following inner product:
\[
(x_1,y_1,z_1)(x_2,y_2,z_2)=(x_1+x_2,y_1+y_2,z_1+z_2+\frac{1}{2}(x_1 y_2-x_2 y_1)).
\]
Considering the canonical basis $\left\lbrace e_1,e_2,e_3\right\rbrace$ of $\R^3$ then 
\[
E_1=\frac{\partial}{\partial x}-\frac{y}{2}\frac{\partial}{\partial z},\quad E_2=\frac{\partial}{\partial y}+\frac{x}{2}\frac{\partial}{\partial z},\quad E_3=\frac{\partial}{\partial z}
\]
are the only left invariant vector fields verifying $(E_k)_0=e_k$, $k\in \left\lbrace 1,2,3\right\rbrace$. The canonical metric of ${\rm Nil}_3$ is defined as the only left invariant metric for which $\left\lbrace e_1,e_2,e_3\right\rbrace$ is an orthonormal basis and, for this metric, $\left\lbrace E_1,E_2,E_3\right\rbrace$ has to be an orthonormal basis of vector fields of $\R^3$. 

We can give an isomorphism (which is also a isometry) between this model of $\text{Nil}_3$ and the semidirect product $\R\rtimes_A \R$ with $A=\left(\begin{array}{cc}
   0  & 1 \\
   0  & 0
\end{array}\right)$ (and therefore, it identifies ${\rm Nil}_3$ and $\R^2\rtimes_A\R$ as metric Lie groups) as follows:
\[
\begin{array}{cccl}
\Xi: & {\rm Nil_3}&\rightarrow &\R^2\rtimes_A\R \\
& (x,y,z) &\mapsto & \left(z+\frac{xy}{2},y,x\right).
\end{array}
\]
The minimal surfaces described in Theorem 6 of \cite{FMP} are either vertical planes or they are given through the equation
\[
z(x,y)=\frac{xy}{2}-c\left[\frac{y\sqrt{1+y^2}}{2}+\frac{1}{2}\ln\left(y+\sqrt{1+y^2}\right)\right].
\]
Using $\Xi$, the parametrization $\left(s,u,z(s,u)\right)$, $s,u\in \R$ is mapped into
\[
\left(s u -c\left[\frac{u\sqrt{1+u^2}}{2}+\frac{1}{2}\ln\left(u+\sqrt{1+u^2}\right)\right],u,s\right)\in \R^2\rtimes_A\R.
\]
The generating curve contained in the plane $z=0$ of this surface is
\[
\left(-c\left[\frac{u\sqrt{1+u^2}}{2}+\frac{1}{2}\ln\left(u+\sqrt{1+u^2}\right)\right],u,0\right)\in \R^2\rtimes_A\R,
\]
which generates \eqref{sol2} if $c=0$ and \eqref{yt} and \eqref{xt} when $c\not=0$ after using the change of variables $u=y(t)$ and the equality $\text{arcsinh}(x)=\ln\left(x+\sqrt{1+x^2}\right)$ for all $x\in\R$. Vertical planes in $\text{Nil}_3$ are mapped by $\Xi$ into the vertical planes generated by \eqref{sol1}.

\begin{figure}[h]
    \centering
    \includegraphics[scale=0.49]{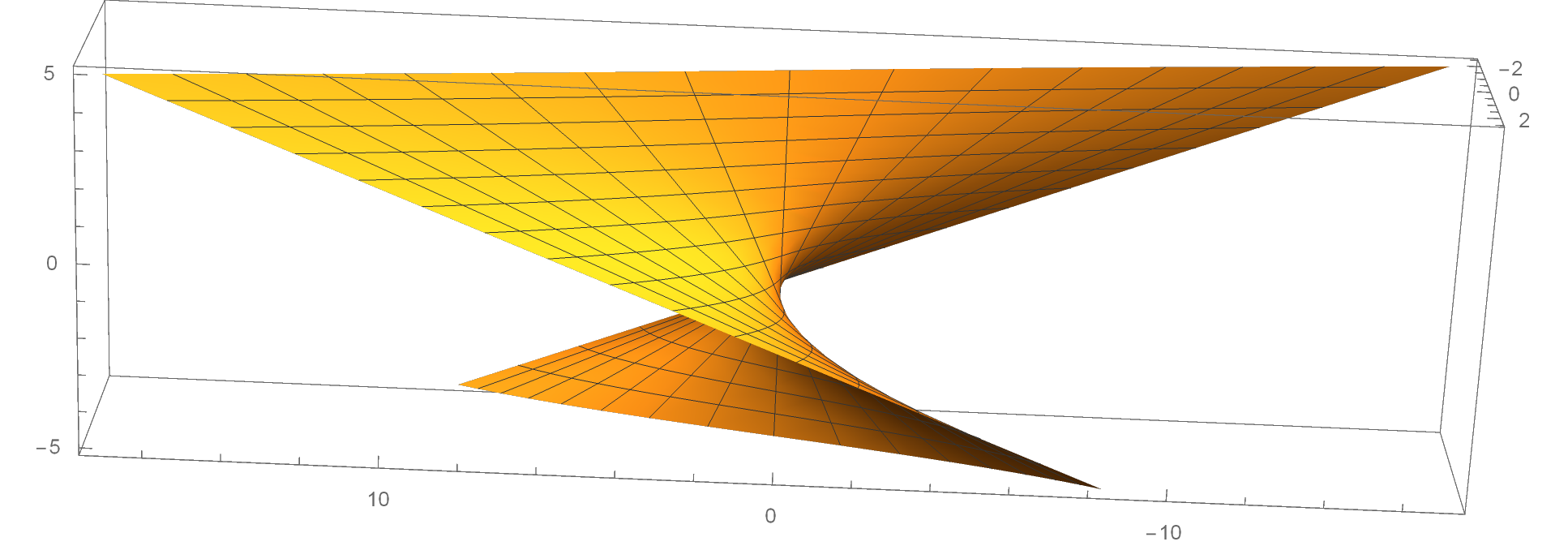}
    \caption{Vertically invariant minimal surface in $\text{Nil}_3$ generated through $\Phi(t,s)=\phi_s(\gamma(t))$, with $\gamma(t)=(x(t),y(t),0)$, where $x,y$ are solutions of the system \eqref{sistemanilbis} with initial conditions $x(0)=y(0)=0$ and $a=1$.}
    \label{fig1}
\end{figure}

\item The solution \eqref{sol1} in Proposition \ref{prop1} gives through the parametrization $\Phi(t,s)=\phi_s(\gamma(t))$, described in \eqref{7}, a vertical plane. All solutions of \eqref{sistema24} that are not in Proposition~\ref{prop1} look similar to the one in Figure \ref{fig1}.

\item When $a\to +\infty$, from the equation $(1+y^2)\cot\theta =a$ we have that $\theta\to 0$ uniformly. Therefore the family of solutions described by \eqref{yt} and \eqref{xt} converge to the solution \eqref{sol1} in Proposition \ref{prop1} when $a\to +\infty$.
\end{enumerate}
}
\end{remark} 

\section{Vertically invariant minimal surfaces in $\widetilde{E}(2)$}\label{sec6}

 If $G=\widetilde{E}(2)$ with its standard metric (isometric to the flat $\R^3$), which is isometric and isomorphic to $\R^2\rtimes_A\R$ with $A=\left(\begin{array}{cc}
0 & -1 \\
1 & 0
\end{array}\right)$, equation \eqref{21} can be written as
\begin{equation}\label{25}
H=\frac{x\sin\theta-y\cos\theta-(1+x^2+y^2)\theta'}
{2(1+(x\cos\theta+y\sin\theta)^2)^{\frac{3}{2}}}
    \end{equation}
When $G=\widetilde{E}(2)$, the orbit of a point $(x,y,0)\in G$ is
\[
\phi_s(x,y,0)=(x \cos s-y\sin s,x\sin s+y\cos s,s),
\]
$s\in \R$, which describes a helicoid. Recall that $\widetilde{E}(2)$ is endowed with its standard flat metric, therefore, vertically invariant CMC surfaces in $\widetilde{E}(2)$ are a subset of the classical CMC surfaces in $\R^3$ invariant through a screw motion. In particular the surfaces of this section are minimal surfaces of $\R^3$. Again by imposing $H=0$, equation~\eqref{25} is written as:
\begin{equation}\label{e2}
\theta'=\frac{x\sin\theta-y\cos\theta}{1+x^2+y^2}
\end{equation}
and the system \eqref{sistema} becomes 
\begin{equation}\label{se2}
\left( \begin{array}{c}
	x \\
	y \\
	\t \end{array}\right)'
=\left( \begin{array}{c}
\cos \theta\\
\sin \theta \\
\frac{x\sin\theta-y\cos\theta}{1+x^2+y^2}		
\end{array}		\right).
\end{equation}
\eqref{se2} is a system of first order ODEs in the three-dimensional manifold $\R^2\times \esf^1=\{ (x,y,\t)\ | \ x,y\in \R,\t\in \esf^1\}$ that also can be viewed as a second order system in~$\R^2$:
\[
x''=(\cos\t)'=-\t'\sin \t
=-\frac{x\sin\theta-y\cos\theta}{1+x^2+y^2}y',
\qquad 
y''
=\frac{x\sin\theta-y\cos\theta}{1+x^2+y^2}x'.
\]

\subsection{Type I solutions}\label{sub61}
Given a fixed $\theta\in \esf^1$, the straight line $(x,y)(t)=t(\cos\t,\sin \t)$, $t\in \R$, is a solution of \eqref{se2}. By moving $\t\in \esf^1$ we find all solutions of~\eqref{se2} with $\g(0)=(0,0)$. If we fix $(x_0,y_0)\in \R^2\setminus \{ (0,0)\}$ then we can not give explicitly the solutions of \eqref{se2} with $\g(0)=(x_0,y_0)$, unless the condition $\g'(0)=\frac{\g(0)}{\| \g(0)\|}$ is verified (in this situation $\g$ is a reparametrization of the straight line joining $(0,0)$ and $(x_0,y_0)$). A solution of~\eqref{se2} will be called of type I if it parameterizes an affine line passing through the origin. The solutions of~\eqref{se2} that are not of type I will be called type II solutions.

\subsection{Solutions of type II}\label{sub62}

The right-hand side of~\eqref{e2} can be written as
\[
\frac{x\sin\theta-y\cos\theta}{1+x^2+y^2}=
\frac{\langle J\g,\g'\rangle }{1+|\g|^2},
\]
where $J(x,y)=(-y,x)$.
This leads us to write~\eqref{se2} as
\begin{equation}\label{se3}
\left( \begin{array}{c}
\g \\
\g' \end{array}\right)'
=\left( \begin{array}{c}
		x' \\
		y'\\
\frac{\langle J\g,\g'\rangle }{1+|\g|^2}J\g'
\end{array}\right). 	
\end{equation}
This means that, if we call $\kappa $ to the curvature of $\g$ with respect to the flat metric of the plane $\left\lbrace z=0\right\rbrace$, our ODE becomes $\kappa J\g'=\g''=\frac{\langle J\g,\g'\rangle }{1+|\g|^2}J\g'$, or equivalently
\begin{equation}\label{se4}
\kappa=\frac{\langle J\g,\g'\rangle }{1+|\g|^2}
=-\frac{\xi}{1+|\g|^2},
\end{equation}
where $\xi=\langle \g,J\g'\rangle $ is the support function of $\g$. The next theorem gives a complete description of solutions of \eqref{se4}.
\begin{theorem}\label{lema5}
If $\g$ verifies~\eqref{se4}, then
\begin{enumerate}
\item $|\kappa(t)|\leq 1/2$, $\forall t\in \R$.
\item If $\g$ intersects the origin $(0,0)$, then $\g$ is type I. Otherwise, $\g$ is a strictly convex curve, i.e. $\kappa$ is a nowhere vanishing function. In particular $\kappa$ has constant sign.
\item The maximal interval of definition of $\g$ is $\R$, and $\kappa (t)$ goes to 0  when $t\to 
\pm \infty$.
\item The curvature $\kappa$ verifies the ODE $\langle \g,\g'\rangle\kappa +(1+|\g|^2)\kappa'=0$.

\item Given $A\in SO(2)$, the curve $A\g$ is also a solution of \eqref{se4}.


\item If $\g$ does not pass through the origin, then there exists a unique $t_0\in \R$ such that 
$\kappa'(t_0)=0$. Moreover:
\begin{itemize}
\item The distance from $\g$ to the origin reaches its unique critical point at $\g(t_0)$, which is a global minimum. 
\item The image of $\g$ is symmetric with respect to the reflection through the normal line to $\g$ at $t_0$.
\item The support function $\xi$ of $\g$ has a unique critical point at $t_0$. 
\item All the self-intersection points of $\g$ occur for opposite values of its parameter (i.e. $\g(t_1)=\g(t_2)$ $\Rightarrow$ $t_1=\pm t_2$) and they are all in the normal line $r$ to $\g$ at $t_0$. Moreover the angle of intersection between $r$ and $\g$ at $t$ is $\pi/2$ if and only if $t=t_0$.
\end{itemize} 
\end{enumerate} 
\end{theorem}
\begin{proof}
From~\eqref{se4} and the Schwarz inequality we have that
\begin{equation}\label{36}
|\kappa|\leq \frac{|\xi|}{1+|\g|^2}\leq 
\frac{|\g|}{1+|\g|^2},
\end{equation}
which proves item 1.

If $\kappa (t_0)=0$ for some $t_0\in \R$, then using~\eqref{se4} we have $\xi(t_0)=0$ . Therefore, either $\g(t_0)=(0,0)$ (in this case $\g$ is of type I) or $\g(t_0)\neq (0,0)$ is orthogonal to $J\g'(t_0)$, and therefore $\g(t_0)$ is collinear with $\g'(t_0)$. In this case, the uniqueness of the initial value problem associated to~\eqref{se3} implies that $\g$ is of type I. This proves item 2.

We can view~\eqref{se3} as an ODE of the following form: $X'=F(t,X)$ with $X=(\g,\g')\in \R^4$. The fact that $F$ is bounded in our equation (because $\g$ is arc length parameterized and because of item 1) and Picard-Lindelöf Theorem ensure that the maximal interval of definition of $X$ is $\R$.

$\g$ cannot be bounded, otherwise the surface $\Sigma\subset \R^2\rtimes_A\R$ generated by $\g$ from~\eqref{7} would be inside a vertical straight cylinder $\mathbb{D}(R)\times \R\subset \R^3$, where $\mathbb{D}(R)$ is the open disc of radius $R>0$ in $\R^2$. Since $\R^2\rtimes_A\R$ is isometric to $\R^3$ with its standard metric, we would reach a contradiction by applying the maximum principle at infinity (see \cite{MeeksRosenberg}) to $\Sigma$ and a vertical plane that keeps $\mathbb{D}(R)\times \R$ on one side. Since $|\g(t)|\to 
\infty $ when $|t|\to \infty$, from~\eqref{36} we conclude that $\kappa(t)\to 0$ when $t\to \pm \infty$.

By cross multiplying in~\eqref{se4} and taking derivatives we get
\begin{equation}\label{inyect}
    \begin{split}
0&=[(1+|\g|^2)\kappa+\xi]'=2\langle \g,\g'\rangle
\kappa+(1+|\g|^2)\kappa '+(\langle \g,J\g'\rangle )'\\
&=2\langle \g,\g'\rangle
\kappa+(1+|\g|^2)\kappa '+\langle \g,-\kappa \g'\rangle =\langle \g,\g'\rangle
\kappa+(1+|\g|^2)\kappa ',
\end{split}
\end{equation}
which proves item 4.

To prove item 5, let $A\in SO(2)$, and $\wt{\g}=A\g$. The curvature $\wt{\kappa}$ of $\wt{\g}$ is $\wt{\kappa}=\kappa$, and $\langle J\wt{\g},\wt{\g}'\rangle =\langle JA\g, A\g'\rangle =\langle AJ\g,A\g'\rangle =\langle J\g, \g'\rangle $ therefore $\wt{\g}$ satisfies~\eqref{se3}.

To prove item 6: due to item 2 we can suppose (up to orientation) that the curvature of $\g$ is positive. Using item 3 we know that $\kappa\colon \R\to (0,\infty)$ reaches a maximum for some $t_0\in \R$. Since $\kappa'(t_0)=0$ and $\kappa(t_0)\neq 0$, item 4 implies that the position vector $\g(t_0)$ is orthogonal to $\g'(t_0)$ (both vectors are nonzero). Item 5 implies that up to a rotation around the origin, we can assume that $\g(t_0)=(x(t_0),0)$ and therefore, $x'(t_0)=0$, $y'(t_0)=\pm 1$. Also, without loss of generality, using a translation of the parameter $t$, we can suppose that $t_0=0$.

We now prove that $\g$ is symmetric with respect to the real axis. Let $t\in \R\mapsto \wt{\g}(t)=(x(t),-y(t))$ be the reflection of $\g$ through the real axis. The curvature $\wt{\kappa}$ of $\wt{\g}$ is $\wt{\kappa}=-\kappa$, and 
\[
\langle J\wt{\g},\wt{\g}'\rangle =\langle J(x,-y),(x',-y')\rangle =\langle (y,x),(x',-y')\rangle =x'y-xy'=-\langle J\g,
\g'\rangle ,
\]
therefore $\wt{\g}$ satisfies~\eqref{se3}. Since $\wt{\g}(t_0)=\g(t_0)$ and $\wt{\g}'(t_0)=-\g'(t_0)$, we conclude that $\g$ is symmetric through the reflection in the real axis. In particular, $\g$ intersects orthogonally the real axis at $\g(0)$.

In order to prove that $\kappa'$ only vanishes at $t_0$ we are going to suppose that $\kappa'(t_1)=0$ for some $t_1\neq 0$. With the same argument as before, $\g$ is symmetric with respect to the reflection of the normal line $r(t_1)$ to $\g$ at $t_1$. In this situation we have three options:
\begin{itemize}
\item $r(t_1)$ is the real axis. In this situation, $\g$ is orthogonal to the real axis at $t_1$, therefore, using the symmetry of $\g$ with respect to the real axis, $\g$ is a closed curve. This contradicts item 3.
\item $r(t_1)$ is parallel to (and different from) the real axis. This implies that $\g$ is symmetric with respect to reflections through two distinct parallel lines, therefore $\g$ is invariant by a non trivial vertical translation obtained by applying both reflections. This implies that $\kappa$ is periodic, which contradicts item 3. 
\item $r(t_1)$ is not parallel to the real axis. Therefore, $\g$ is symmetric with respect to both the real axis and $r(t_1)$, which implies that $\g$ is symmetric with respect to the composition of both reflections which is a rotation of angle $\alpha\neq 0,\pi$. This implies that $\kappa$ is periodic, which contradicts item 3.
\end{itemize}
Thus, $\kappa'$ only vanishes at $t_0=0$. Since we are supposing $\kappa>0$, we have that $\kappa$ reaches its global maximum at $t=0$.

Since, because of item 4, $(|\g|^2)'=2\langle \g,\g'\rangle$ only vanishes at the zeros of $\kappa'$, this implies that the distance to the origin restricted to $\g$ has a unique critical point at $t=0$. This critical point is the global minimum of $|\g|$, because $\g$ is not bounded. 

Now we study the behaviour of each symmetric half $\g^+=\g(t_0,\infty)$, $\g^-=S(\g^+)$ of $\g$ that starts at the point $\g(0)$. In this way, $\g^+$ is a branch of $\g$ which is orthogonal to the real axis at its initial point $t_0=0$. Up to a rotation of angle $\pi$ around $\g(0)$, we can suppose that $y(t)>0$ for $t\in (0,\varepsilon)$ and for a small enough $\varepsilon>0$, which means that $\g^+$ starts lying in the upper half plane; which is equivalent to write $y'(0)=1$. Since $\kappa>0$, $\g^+$ starts being at the left-hand side of the vertical line $\{ x=x(0)\}$. 

In the case we suppose that the height function $\langle \g(t),\g'(0)\rangle =y(t)$ over the real axis has no critical points then $\g^+$ is a convex global graph over its projection $(-\infty ,x(0))$ over the real axis (this graph may be unbounded). Therefore, the unit normal $J\g'=n=(n_1,n_2)$ along $\g^+$ is in one of the semicircumferences defined by $\esf^1\setminus \{ n_2=0\}$. Notice that $\g$ is a global graph with respect to the straight line which is tangent to $\g$ at $\g(0)$ (because $y(t)$ has no critical points). In particular, $\g$ is embedded. This implies that the surface $\Sigma\subset \R^2\rtimes_A\R$ generated by $\g$ through~\eqref{7} is a complete embedded simply-connected minimal surface in $\R^3$ with its standard metric and, therefore, $\Sigma$ is a plane or an helicoid, which is a contradiction.

In conclusion, the height function $y(t)$ over the real axis has at least one critical point at $t_1\in (0,\infty)$. This implies that the tangent line to $\g$ at $t_1$ is horizontal, and since $\kappa(t_1)>0$, $y(t)$ has a local maximum at $t_1$. Using that $\kappa$ can not change sign, we deduce that $\g^+$ has to intersect the real axis at a new point $\g(t_2)$ with $t_2>t_1$. In this situation, $\g'(t_2)$ can not be orthogonal to the real axis (if it were, we would conclude that $\g$ is a closed curve which is a contradiction with item 3). Therefore, $\g(t_2)$ is a point of self-intersection at which $\g^+$ and $\g^-$ intersect transversely as shown in Figure~\ref{fig2}.
\begin{figure}[h]
\centering
\includegraphics[scale=0.7]{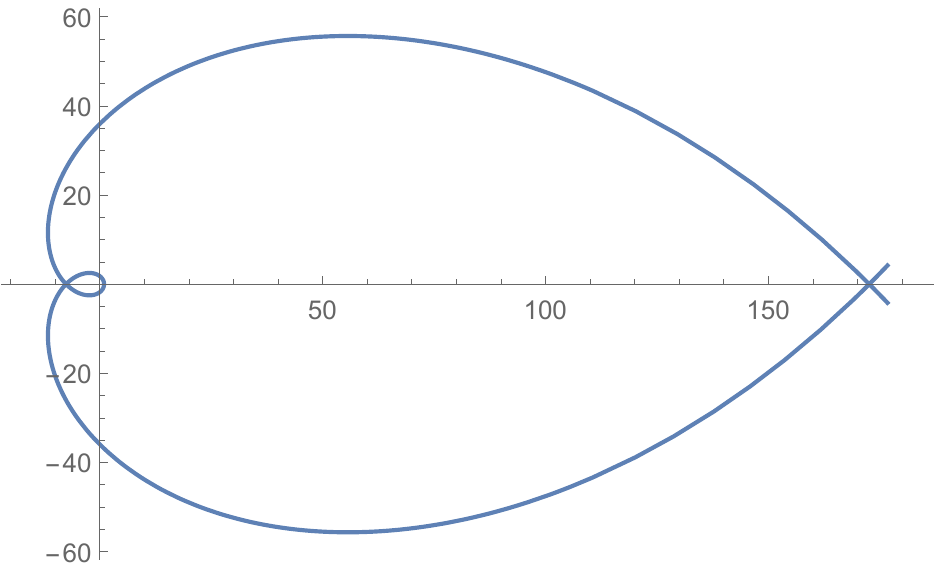}
\caption{The numerical solution $\g$ of~\eqref{se4} 
with the initial conditions $\g(0)=(1,0)$, $\g'(0)=(0,1)$. It shows $\g$ for the values $t\in [-250,250]$. For this interval we found two self-intersecting points.}
\label{fig2}
\end{figure}

Since $\g$ is symmetric with respect to the real axis and $t_0=0$, we have that $x(t)$ is an even function (respectively $y(t)$ is an odd function). Therefore, $x'(t)$ is odd (respectively $y'(t)$ is even), and we deduce that the support function $\xi=\langle \g,J\g'\rangle$ is even. Since $\kappa$ is positive, from~\eqref{se4} we have that $\xi$ is negative. Moreover, 
\[
\xi'=(\langle \g,J\g'\rangle )'=\langle \g,(J\g')'\rangle=-\langle \g,\g'\rangle\kappa =
(1+|\g|^2)\kappa',
\]
where, for the last equality, item 4 is needed. Since $\kappa$ is strictly increasing
on $(-\infty,0)$ (respectively is strictly decreasing on $(0,\infty)$), we have that $\kappa'>0$ on $(-\infty,0)$ (respectively $\kappa'<0$ on $(0,\infty)$), therefore $\xi'>0$ on $(-\infty,0)$ (respectively $\xi'<0$ on $(0,\infty)$), which means that $\xi $ is strictly increasing on $(-\infty,0)$ and strictly decreasing on $(0,\infty)$, and $\xi$ has a unique critical point which is a negative global maximum at $t=0$. 

Next, we prove that the map $t\mapsto |\g(t)|$ is inyective on $(0,\infty)$ (the same argument works on $(-\infty,0)$). If there exists $t_1,t_2$ such that $0<t_1<t_2$ and $|\g(t_1)|=|\g(t_2)|$, this implies the existence of $t_3\in (t_1,t_2)$ such that $\xi(t_3)=|\g(t_3)|$. Since $\xi=\left\langle \g,J\g\right\rangle\leq |\g|$, then $\left\langle \g(t_3),\g(t_3)\right\rangle=0$. Using equation \eqref{inyect} we get that $\kappa'(t_3)=0$ which contradicts the fact that $\kappa'$ only vanishes at 0.

\begin{figure}[h]
    \centering
    \includegraphics[scale=0.5]{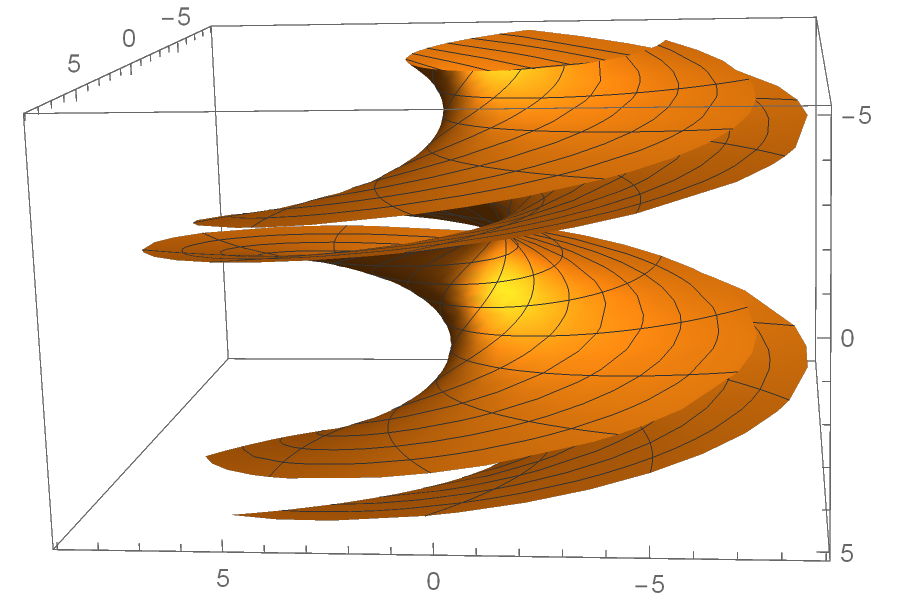}
    \caption{Vertically invariant minimal surface in $\widetilde{E}(2)$ (of type II) generated through $\Phi(t,s)=\phi_s(\gamma(t))$, with $\gamma(t)=(x(t),y(t),0)$, where $x,y$ are solutions of the system \eqref{se2} with initial conditions $x(0)=1$, $y(0)=0$, $\theta(0)=1$.}
    \label{fig3}
\end{figure}

We still have to prove that the self-intersection points of $\g$ are all in the real axis and they occur for opposite values of the parameter of $\g$. If $\g(t_1)=\g(t_2)$ but $t_1\neq t_2$, then $t_1,t_2$ can not have the same sign (because the map $t\mapsto |\g(t)|$ is inyective on $(0,\infty)$ and on $(-\infty,0)$), therefore we can suppose $t_1<0<t_2$. Since $|\g(t_1)|=|\g(t_2)|=|\g(-t_2)|$ (because $|\g|$ is even) then $t_1=-t_2$, and therefore $y(t_2)=y(t_1)=y(-t_2)=-y(t_2)$ which implies $y(t_2)=0$).
\end{proof}

\subsection{Vertically invariant minimal surfaces in $\widetilde{E}(2)$ with a non-flat metric}
We consider $\widetilde{E}(2)$ with the metric that makes it isometric and isomorphic to $\R^2\rtimes_{A(c)}\R$ with 
\[
A(c)=\left(\begin{array}{cc}
	0 & -c \\
	1/c & 0
\end{array}\right)\]
with $c\in (1,\infty)$. If we impose $H=0$ in equation \eqref{26} we get the following system:
\begin{equation}\label{sistemau}
\left\lbrace\begin{array}{l}
x'=\cos \theta\\
y'=\sin\theta\\
\t'=\frac{\left[1+c^2+\left(1-c^2\right) \cos (2 \theta )\right] \left(x \sin\theta -c^2 y \cos\theta \right)}{2 \left(c^2+x^2+c^4 y^2\right)}
\end{array}.\right.
\end{equation}
\begin{proposition}\label{lema7}
If $\g$ verifies \eqref{sistemau} then the maximal interval of definition of $\g$ is $\R$.
\end{proposition}
\begin{proof}
    The function $a\in [-1,1]\mapsto \frac{1}{2}[1+c^2+(1-c^2)a]$ is nonincreasing because $c\geq 1$, therefore its maximum is $\frac{1}{2}[1+c^2-(1-c^2)]=c^2$. Moreover, using Schwarz inequality, 
    \begin{equation}\label{39}
    \left|\frac{x\sin\t -c^2y\cos\t}{c^2+x^2+c^4y^2}\right|\leq \frac{\sqrt{x^2+c^4 y^2}}{c^2+x^2+c^4 y^2}\leq \frac{1}{2c}.
    \end{equation}
    This proves that the derivative of $\t$ is bounded. Again, \eqref{sistemau} can be seen as an ODE of the form $X'=F(t,X)$, where $F$ is bounded and the result follows from the Picard-Lindelöf Theorem.
\end{proof}

Viewing \eqref{sistemau} as an ODE of the form $X'=F(t,X,c)$ for all $c\in [1,\infty)$, since for each $c\in[1,\infty)$ the maximal interval of definition of the solution $\gamma_c$ of \eqref{sistemau} (once we have fixed the initial conditions) is $\R$, and the map $F:\R\times \R^3\times [1,\infty)\rightarrow \R^3$, $(t,X,c)\mapsto F(t,X,c)$ is locally Lipschitz (in fact is differentiable), we can use the continuous dependence on the parameter to conclude that, once we fix the initial conditions, for every $\varepsilon>0$ and for every interval $I\subset \R$ there exists $\delta>0$ such that if $1\leq c<1+\delta$ then $|\gamma(t)-\gamma_c(t)|<\varepsilon$ for every $t\in I$, where $\gamma$ is the solution of \eqref{sistemau} with $c=1$ which have been described in Subsections \ref{sub61} and \ref{sub62}.

Next result describes further properties of the generating curve $\gamma$ in $\widetilde{E}(2)$ with a non-flat metric.

\begin{proposition}
    \begin{enumerate}
        \item If $c>1$ then the x-axis and the y-axis are the only solution of \eqref{sistemau} with constant $\theta$.
        \item The curvature $\kappa=\theta'$ goes to $0$ when $t\to \pm \infty$ and $|\kappa(t)|\leq c/2$, $\forall t\in \R$.
    \end{enumerate}
\end{proposition}
\begin{proof}
    To prove item 1, notice that solving the equation $1+c^2+(1-c^2)C=0$ gives $C=\frac{c^2+1}{c^2-1}$ which is strictly greater than 1. Therefore the factor $1+c^2+(1-c^2)\cos(2\theta)$ can not be 0. Thus, a solution of \eqref{sistemau} with constant $\theta=\theta_0\in \R$ must verify
    \[
    x \sin\theta_0-c^2 y \cos\theta_0\equiv 0,
    \]
    which means that the generating curve $\gamma$ is contained in a straight line passing through the origin. Write $\gamma$ as $\gamma(t)=t v$, where $v=(v_1,v_2)\in \mathbb{S}^1$, for all $t\in \R $ and therefore we have $t v_1\sin\theta_0-c^2 tv_2\cos\theta_0=t(-1+c^2)v_1v_2=0$ for all $t\in \R$. As we are assuming $c>1$, then $v_1 v_2=0$.

    To prove item 2, we know that $|\kappa(t)|\leq c/2$, $\forall t\in \R$ because of the proof of Proposition~\ref{lema7}. To see that $\kappa=\theta'$ goes to $0$ when $t\to \pm \infty$, the same proof of Theorem~\ref{lema7} does not work here. However, vertical planes are still minimal surfaces in $\R^2\rtimes_A\R$, for any $A\in \mathcal{M}_2(\R)$ (see Remark 2.10 in \cite{MeeksPerez}). Suppose that $\gamma$ is bounded, then the surface $\Sigma\subset\widetilde{E}(2)$ generated by $\gamma$ from \eqref{7} would be inside a vertical cylinder $\mathbb{D}(R)\times\R\subset\R^3$, $R>0$. We can take a vertical plane disjoint from $\Sigma$ and left translate it until we find a first contact point. If the first contact point is at infinity (see Remark \ref{r9}) then there is also a first finite contact point. This comes from the fact that the vertical translation of a point $(x,y,z)\in\widetilde{E}(2)$ describes a helix $s\mapsto \left(x\cos s-cy\sin s,\frac{x\sin s}{c}+y\cos s,z+s\right)$. Therefore the surface $\Sigma$ is periodic along the vertical axis, which implies that there are infinitely many finite contact points when we find a first contact point. Thus, we can apply the classical maximum principle to reach a contradiction and therefore $|\gamma(t)|\to \infty$ when $t\to \pm\infty$. Using \eqref{39} we have
    \[
    \kappa\leq c^2 F\left(\sqrt{x^2+c^4y^2}\right),
    \]
    where $F(b)=\frac{b}{c^2+b^2}$. Since $\sqrt{x^2+c^4y^2}\to \infty$ if $|\gamma|\to \infty$ and $F(b)$ goes to $0$ when $b\to\infty$ we have that $k(t)\to 0$ if $|t|\to\infty$.
\end{proof}

\begin{remark}\label{r9}
    It is still not proved a maximum principle at infinity result or a half-space theorem for vertical planes in $\widetilde{E}(2)$. However a half-space theorem for horizontal planes was proved in {\rm\cite{zang}}.
\end{remark}

\section{Surfaces with zero Gaussian curvature}\label{sec7}
Since in Section \ref{sec3} we have computed the coefficients of the first and second fundamental forms, it is natural to compute the Gaussian curvature of a surface generated by \eqref{7} and study the vertically invariant surfaces on $\R^2\rtimes_A\R$ with constant or zero Gaussian curvature. 

Again, we assume that our metric Lie group $G=\R^2\rtimes_A\R$, $A\in \mathcal{M}_2(\R)$ is unimodular. In the literature, invariant surfaces of zero Gaussian curvature in metric Lie groups that can be written as a semidirect product has been studied in the following ambient spaces:
\begin{enumerate}
    \item In $\R^3$, Hartman and Nirenberg proved that (Theorem III in \cite{hartnir}) all the complete, zero Gaussian curvature surfaces are given by right cylinders over possibly non compact curves in $\R^2\times\{0\}$, which describes all the possible examples.
    
    \item The invariant surfaces with constant extrinsic Gaussian curvature in $\text{Nil}_3$ were studied by Belarbi in \cite{belarbi}, where he considered all possible 1-parameter subgroups of left translations.
    
    \item Invariant surfaces with respect to any 1-parameter subgroup of constant intrinsic or extrinsic Gaussian curvature in $\text{Sol}_3$ with its standard metric, $A=\left(\begin{array}{cc}
        1 & 0 \\
        0 & -1
    \end{array}\right)$, were studied in López \cite{loraf1} and López-Munteanu \cite{loraf2}.
\end{enumerate} 

When we consider $\widetilde{E}(2)$ with its standard (flat) metric, we know, by Theorem III in \cite{hartnir} that except for cylinders over a circumference centered at the origin, there are no complete zero Gaussian curvature invariant surfaces because vertical translations in $\widetilde{E}(2)$ do not produce right cylinders except for these vertical cylinders over circumferences. Even though, we can still give a description of the vertically invariant surfaces in $\widetilde{E}(2)$ with zero Gaussian curvature (since the metric is flat we are dealing with both, intrinsic and extrinsic, cases). None of these surfaces can be complete except for the vertical cylinders over circumferences centered at the origin.

Using the expressions for the first and second fundamental forms of Section \ref{sec3}, and the same notation as in previous sections, since the Gaussian curvature $K$ can be computed as
\[
K=\frac{eg-f^2}{EG-F^2},
\]
we get the following system of ODEs for a curve $\gamma=(x,y,0)$ that generates, through vertical translations, a constant Gaussian curvature surface in the standard $\widetilde{E}(2)$:
\begin{equation}
\left\lbrace\begin{array}{l}
x'=\cos \theta\\
y'=\sin\theta\\
K=-\frac{1+(y\cos \theta -x\sin \theta)\theta'}{(1+(x \cos\theta+ y \sin\theta))^2}
\end{array}.\right.
\label{sistemagauss}
\end{equation}
If we impose $K=0$ this system becomes:
\begin{equation}
\left\lbrace\begin{array}{l}
x'=\cos \theta\\
y'=\sin\theta\\
(y\cos \theta -x\sin \theta)\theta'=-1
\end{array}.\right.
\label{sistemagauss2}
\end{equation}
As in Section \ref{sec5}, we can give a first integral for the system \eqref{sistemagauss2}:
\[
J(t)=x(t)\cos\theta(t)+y(t)\sin\theta(t)=:J\in \R.
\]
$J$ can be written as
\[
J(t)=\frac{1}{2}\frac{d}{dt}\left(x(t)^2+y(t)^2\right).
\]
Integrating and assuming $x(0)=x_0$, $y(0)=y_0$, $\theta(0)=\theta_0$, $x_0,y_0,\theta_0\in \R$ gives that for any $t$ in the maximal interval of definition of $\gamma$ we have
\begin{equation}\label{modulo}
|\gamma(t)|^2=x(t)^2+y(t)^2=a t+x_0^2+y_0^2,
\end{equation}
where $a:=2J=2(x_0\cos\theta_0+y_0 \sin\theta_0)$. We know that if $a\not=0$ there are no solutions of \eqref{sistemagauss2} defined for all $t\in \R$. \eqref{modulo} implies that the maximal interval of definition of $\gamma$ is contained in $\left[-(x_0^2+y_0^2)/a,\infty\right)$ if $a>0$ or in $\left(-\infty,-(x_0^2+y_0^2)/a\right]$ if $a<0$. Now write $\gamma$ in polar coordinates:
\[
x(t)=r(t)\cos\alpha(t), \quad
y(t)=r(t)\sin \alpha(t).
\]
Due to \eqref{modulo} we know that 
\begin{equation}\label{rt}
r(t)=\sqrt{a t +x_0^2+y_0^2}.
\end{equation}
Since 
\begin{align*}
    x'&=r'\cos\alpha-r\alpha'\sin\alpha,\\
    y'&=r'\sin \alpha+r\alpha'\cos\alpha,
\end{align*}
equation $1=(x')^2+(y')^2$ becomes
\[
1=(r')^2+r^2(\alpha')^2,
\]
which implies
\[
\alpha'(t)=\frac{1}{2} \frac{\sqrt{-a^2+4 a t+4 \left(x_0^2+y_0^2\right)}}{\left|a t+x_0^2+y_0^2\right|}.
\]
The maximal interval of definition of $\alpha'$ (and therefore the maximal interval of definition of $\alpha$) is $I_1:=\left[a/4-(x_0^2+y_0^2)/a,\infty\right)$ if $a>0$, $I_2:=\left(-\infty,a/4-(x_0^2+y_0^2)/a\right]$ if $a<0$ and $\R$ if $a=0$. Since $I_1\subset\left[-(x_0^2+y_0^2)/a,\infty\right)$ when $a>0$ (resp. $I_2\subset\left(-\infty,-(x_0^2+y_0^2)/a\right]$) then the maximal interval of definition of a solution of \eqref{sistemagauss2} is given by $I_1$ if $a>0$, $I_2$ if $a<0$, or $\R$ if $a=0$.
 \begin{figure}[h]
\centering
\begin{subfigure}{.5\textwidth}
\centering
\includegraphics[width=.8\linewidth]{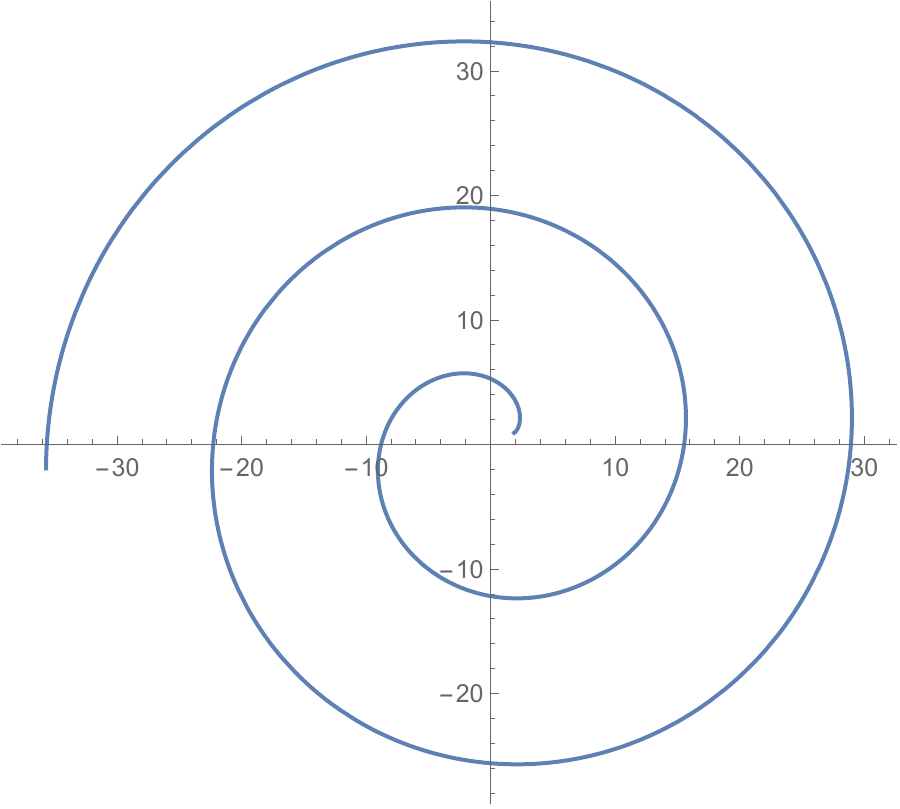}
\label{fig:sub1}
\end{subfigure}%
\begin{subfigure}{.5\textwidth}
\centering
\includegraphics[width=.8\linewidth]{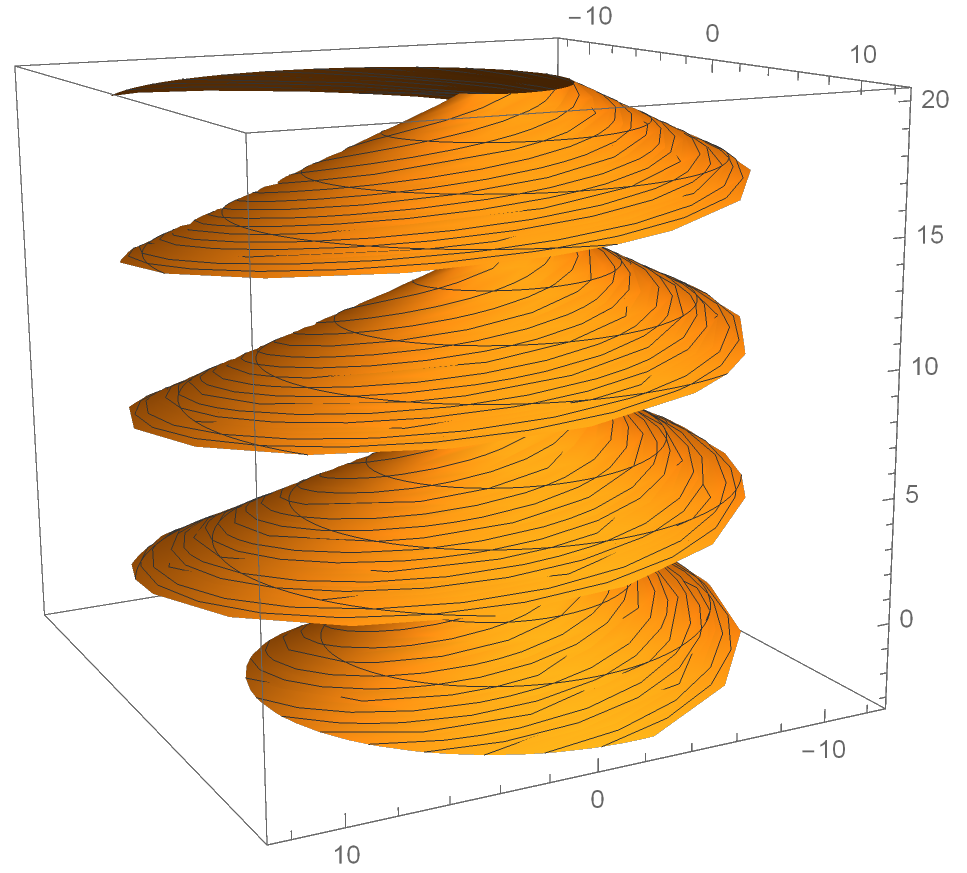}
\label{fig:sub2}
\end{subfigure}
\caption{Left: solution $\gamma$ of \eqref{sistemagauss2} with initial conditions $x_0=2$, $y_0=1$, $\theta_0=\frac{\pi}{4}$ ($a>0$). Right: zero Gaussian curvature surface generated through \eqref{7} with this $\gamma$.}
\end{figure}

Set 
\[
g(t)=\sqrt{-a^2+4at+4(x_0^2+y_0^2)},
\]
then, integrating, we finally describe the function $\alpha$ and, therefore, all the invariant surfaces in $\widetilde{E}(2)$ with zero Gaussian curvature:
\begin{equation}
\begin{split}\label{at}
    \alpha(t)=&\arctan\left(\frac{y_0}{x_0}\right)+\frac{1}{2}\int_0^t \frac{\sqrt{-a^2+4 a \tau+4 \left(x_0^2+y_0^2\right)}}{\left|a \tau+x_0^2+y_0^2\right|}d\tau\\
    =&\arctan\left(\frac{y_0}{x_0}\right)+\left.\frac{\text{sgn}\left(a\tau+x_0^2+y_0^2\right)}{2a}\left(g(\tau)-a\;\text{arccot}\left(\frac{a}{g(\tau)}\right)\right)\right|_{\tau=0}^{\tau=t}.
\end{split}
\end{equation}

In summary, we have proved the following result:
\begin{proposition}
    Let $\Sigma\subset \widetilde{E}(2)$ be a vertically invariant surface with zero Gaussian curvature, then either $\Sigma$ is a vertical cylinder over a circumference in the plane $\left\lbrace z=0\right\rbrace$ centered at the origin (and this is the only complete example) or the generating curve $\gamma$ of $\Sigma$ is given in polar coordinates by \eqref{rt} and \eqref{at}.
\end{proposition}

\bibliographystyle{plain}
\bibliography{main}

\Addresses

\end{document}